\documentclass[12pt,reqno]{amsart}

\usepackage{amssymb,amsthm}
\usepackage{amsfonts,cmtiup,comment,stmaryrd}

\usepackage{mathrsfs,cmtiup}

\makeatletter
\def\blfootnote{\xdef\@thefnmark{}\@footnotetext}
\makeatother

\newtheorem{theorem}{Theorem}[section]
\newtheorem{lemma}[theorem]{Lemma}
\newtheorem{proposition}[theorem]{Proposition}

\theoremstyle{definition}

\newtheorem{remark}[theorem]{Remark}
\newtheorem*{definition*}{Definition}

\newcommand{\N}{\mathbb N}

\newcommand{\Z}{\mathbb Z}
\newcommand{\Q}{\mathbb Q}

\newcommand{\f}{\varphi}
\newcommand{\e}{\varepsilon }

\newcommand{\g}{\gamma }

\renewcommand{\geq}{\geqslant}
\renewcommand{\leq}{\leqslant}

\newcommand{\ed} {\end{document}}

\let\leq=\leqslant
\let\geq=\geqslant
\numberwithin{equation}{section}

\begin{document}
\title{On profinite groups with automorphisms\\[5pt] whose fixed points have countable  Engel sinks}

\author{E. I. Khukhro}
\address{Charlotte Scott Research Centre for Algebra, University of Lincoln, U.K., and \newline \indent  Sobolev Institute of Mathematics, Novosibirsk, 630090, Russia}
\email{khukhro@yahoo.co.uk}

\author{P. Shumyatsky}

\address{Department of Mathematics, University of Brasilia, DF~70910-900, Brazil}
\email{pavel@unb.br}


\begin{abstract}
An Engel sink of an element $g$ of a group $G$ is a set ${\mathscr E}(g)$ such that for every $x\in G$ all sufficiently long commutators $[...[[x,g],g],\dots ,g]$ belong to ${\mathscr E}(g)$.  (Thus, $g$ is an Engel element precisely when we can choose ${\mathscr E}(g)=\{ 1\}$.) It is proved that if a profinite group $G$ admits an elementary abelian group of automorphisms $A$ of coprime order $q^2$ for a prime $q$ such that for each $a\in A\setminus\{1\}$ every element of the centralizer $C_G(a)$ has a countable (or finite) Engel sink, then $G$ has a finite normal subgroup $N$ such that $G/N$ is locally nilpotent.
\end{abstract}
\maketitle

\section{Introduction}
It is well known that if a finite group $G$ admits a non-cyclic abelian group of automorphisms $A$ of coprime order, then $G=\langle C_G(a)\mid a\in A^\#\rangle$, where $A^\#= A\setminus \{1\}$ (see, for example, \cite[Theorem~6.2.4]{gor}). Therefore many properties of the group $G$ can be derived from the properties of the centralizers $C_G(a)$. For example, we proved in \cite{khu-shu99} that the exponent of $G$ is bounded in terms of the exponents of the $C_G(a)$ and $|A|$. In most cases it is natural to assume that $A$ is an elementary abelian group of order $q^2$ for a prime $q$.

Results of this kind have also been recently extended to profinite groups admitting a non-cyclic abelian finite group acting by coprime automorphisms \cite{acc-shu16,
acc-shu-sil18,
acc-shu-sil19,
acc-sil18,
acc-sil20}. In particular,
Acciarri, Shumyatsky, and Silveira \cite{acc-shu-sil18} proved the following.

\begin{theorem}[{\cite[Theorem~B2]{acc-shu-sil18}}]\label{t-ass}
If a profinite group $G$ admits an elementary abelian $q$-group $A$ of order $q^2$ acting by coprime automorphisms on $G$ in such a way that all elements in $C_G(a)$  are Engel in $G$ for every $a\in A^\#$, then $G$ is
locally nilpotent.
\end{theorem}

This result can be viewed as an `automorphism' extension of the Wilson--Zelmanov theorem \cite[Theorem~5]{wi-ze} saying that a profinite Engel group is locally nilpotent. This Wilson--Zelmanov theorem was based on Zelmanov's deep results \cite{ze92,ze95,ze17} on Lie algebras with Engel conditions, and the proof in the aforementioned `automorphism' extension also used these results.

Recall that a group $G$ is called an Engel group if for every $x,g\in G$ the equation $[x,\,{}_{n} g]=1$ holds for some $n=n(x,g)$ depending on $x$ and $g$.
Henceforth, we use the left-normed simple commutator notation
$[a_1,a_2,a_3,\dots ,a_r]:=[...[[a_1,a_2],a_3],\dots ,a_r]$ and the abbreviation $[a,\,{}_kb]:=[a,b,b,\dots, b]$ where $b$ is repeated $k$ times. A group is said to be locally nilpotent if every finite subset generates a nilpotent subgroup.

In recent papers \cite{khu-shu16,
khu-shu18,
khu-shu18a,
khu-shu19,
khu-shu20,
khu-shu-tra} we considered generalizations of Engel conditions for finite, profinite, and compact groups  using the concept of Engel sinks.

\begin{definition*} \label{d}
 An \textit{Engel sink} of an element $g$ of a group $G$ is a set ${\mathscr E}(g)$ such that for every $x\in G$ all sufficiently long commutators $[x,g,g,\dots ,g]$ belong to ${\mathscr E}(g)$, that is, for every $x\in G$ there is a positive integer $n(x,g)$ such that
 $[x,\,{}_{n}g]\in {\mathscr E}(g)$ for all $n\geq n(x,g).
 $
 \end{definition*}
 \noindent (Thus, $g$ is an Engel element precisely when we can choose ${\mathscr E}(g)=\{ 1\}$, and $G$ is an Engel group when we can choose ${\mathscr E}(g)=\{ 1\}$ for all $g\in G$.)

In \cite{khu-shu18,khu-shu20} we considered compact (Hausdorff) groups in which  every element has a finite or countable Engel sink and proved the following theorem.

\begin{theorem}[{\cite[Theorem~1.2]{khu-shu20}}]\label{t-20}
If every element of a compact group $G$ has a countable Engel sink, then $G$ has a finite normal subgroup $N$ such that $G/N$ is locally nilpotent.
\end{theorem}
(By ``countable'' we mean ``finite or  denumerable''.)

The purpose of this paper is to obtain the following `automorphism' extension of this result for profinite groups.

\begin{theorem}\label{t-main}
Let $q$ be a prime, and $A$ an elementary abelian $q$-group of order $q^2$. Suppose that $A$ acts  by coprime automorphisms on a profinite group $G$ in such a way that for each $a\in A^\#$ every element of $C_G(a)$ has a countable Engel sink in $G$. Then $G$ has a finite normal subgroup $N$ such that $G/N$ is locally nilpotent.
\end{theorem}

In the proof, first the case of pro-$p$ groups  is considered, where Lie ring methods are applied including Zelmanov's theorem on Lie algebras satisfying a polynomial identity and generated by elements all of whose products are ad-nilpotent
\cite{ze92,ze95,ze17}, in conjunction with the Bahturin--Zaitsev theorem \cite{bah-zai} on polynomial identities of Lie algebras with automorphisms. This analysis provides a reduction to uniformly powerful pro-$p$ groups, for which a different Lie algebra over $p$-adic integers
 is used canonically connected with the group via the Baker--Campbell--Hausdorff formula. There is no straightforward connection between Engel commutators in the group and in the corresponding Lie algebra, or between Engel sinks of their elements. Nevertheless the Lie algebra is used to prove that in a uniformly powerful pro-$p$ group elements with countable Engel sinks are in fact Engel elements. Then the desired result is derived for the case of pronilpotent groups. In the general case, firstly an open locally nilpotent subgroup is found, and the proof proceeds by induction on its index.

\section{Preliminaries}

In this section we recall some  notation and terminology and several known  properties of Engel sinks in profinite groups.

Our notation and terminology for profinite groups is standard; see, for example,  \cite{anal, rib-zal,  wil}.  A subgroup (topologically) generated by a subset $S$ is denoted by $\langle S\rangle$. Recall that centralizers are closed subgroups, while commutator subgroups $[B,A]=\langle [b,a]\mid b\in B,\;a\in A\rangle$ are the closures of the corresponding abstract commutator subgroups.

We denote by $\pi (k)$ the set of prime divisors of $k$, where $k$ may be a positive integer or a Steinitz number, and by $\pi (G)$ the set of prime divisors of the orders of elements of a (profinite) group $G$. Let $\sigma$ be a set of primes. An element $g$ of a group is  a $\sigma$-element if $\pi(|g|)\subseteq \sigma$, and a group $G$ is a $\sigma$-group if all of its elements are $\sigma$-elements. We denote by $\sigma'$ the complement of $\sigma$ in the set of all primes. When $\sigma=\{p\}$,  we write $p$-element, $p'$-element, etc.

Recall that a pro-$p$ group  is an inverse limit of finite $p$-groups, a pro-$\sigma $ group is an inverse limit of finite $\sigma$-groups, a pronilpotent group is an inverse limit of finite nilpotent groups, a prosoluble group is an inverse limit of finite soluble groups.

We denote by  $\gamma _{\infty}(G)=\bigcap _i\gamma _i(G)$ the intersection of the lower central series of a group~$G$. A profinite group $G$ is pronilpotent if and only if $\gamma _{\infty}(G)=1$.

Profinite groups have Sylow $p$-subgroups and satisfy analogues of the Sylow theorems.  Prosoluble groups satisfy analogues of the
theorems
 on Hall $\pi$-subgroups. Pronilpotent groups are Cartesian products of their Sylow $p$-subgroups.
 We refer the reader to the corresponding chapters in \cite[Ch.~2]{rib-zal} and \cite[Ch.~2]{wil}.

For a group $A$ acting by automorphisms on a group $B$ we use the usual notation for commutators $[b,a]=b^{-1}b^a$ and commutator subgroups $[B,A]=\langle [b,a]\mid b\in B,\;a\in A\rangle$, as well as for centralizers $C_B(A)=\{b\in B\mid b^a=b \text{ for all }a\in A\}$
and $C_A(B)=\{a\in A\mid b^a=b\text{ for all }b\in B\}$.

 If $\varphi$  is an automorphism of a finite group $H$ of coprime order, that is, such that $(|\varphi |,|H|)=1$, then  we say for brevity that $\varphi$ is a coprime automorphism of~$H$. Among many well-known properties of coprime automorphisms of finite groups, we recall the following.

\begin{lemma}[{see \cite[Theorem~6.2.4]{gor}}] \label{l-q2}
If $A$ is an elementary abelian group of order $q^2$ for a prime $q$ acting by coprime automorphisms on a finite group $G$, then $G=\langle C_G(a)\mid a\in A^\#\rangle$.
\end{lemma}

We say that $\varphi$ is a \textit{coprime automorphism}  of a profinite group $H$  meaning that a procyclic group $\langle\varphi\rangle$ faithfully acts on $H$ by continuous automorphisms   and $\pi (\langle \varphi\rangle)\cap \pi (H)=\varnothing$. Since the semidirect product $H\langle \varphi\rangle$ is also a profinite group, $\varphi$ is a coprime automorphism of $H$
 if and only if for every open normal $\varphi$-invariant subgroup $N$ of $H$ the automorphism (of finite order) induced by $\varphi$ on $H/N$ is a coprime automorphism. The following folklore lemma follows from the Sylow theory for profinite groups and an analogue of the Schur--Zassenhaus theorem.

\begin{lemma}[{see~\cite[Lemma~4.1]{khu-shu20}}]\label{l-inv}
If $A$ is a group of coprime automorphisms of a profinite group $G$, then for every prime $q\in \pi (G)$ there is an $A$-invariant Sylow $q$-subgroup of $G$. If $G$ is in addition prosoluble, then for every subset $\sigma\subseteq \pi (G)$ there is an $A$-invariant Hall $\sigma$-subgroup of~$G$.
\end{lemma}

The following lemma is a profinite analogue of the well-known property of coprime automorphisms of finite groups.

\begin{lemma}[{\cite[Proposition~2.3.16]{rib-zal}}] \label{l-cover}
If $A$ is a group of coprime automorphisms of a profinite group $G$ and $N$ is an $A$-invariant closed normal subgroup of $G$, then every fixed point of $A$ in $G/N$ is an image of a fixed point of $A$ in $G$, that is, $C_{G/N}(A)=C(A)N/N$.
\end{lemma}

As a consequence, we also have the following.

\begin{lemma}[{\cite[Lemma 4.3]{khu-shu20}}] \label{l-gff}
If $A$ is a group of coprime automorphisms of a profinite group $G$, then
$[[G,A],A]=[G,A]$.
\end{lemma}

The following well-known fact is a straightforward consequence of the Baire Category Theorem (see \cite[Theorem~34]{kel}).

\begin{theorem}\label{bct}
If a compact Hausdorff group is a countable union of closed subsets, then one of these subsets has non-empty interior.
\end{theorem}

We now recall a few general properties of Engel sinks. Clearly, the intersection of two Engel sinks of a given element $g$ of a group $G$ is again an Engel sink of $g$, with the corresponding function $n(x,g)$ being the maximum of the two functions. Therefore, if $g$ has a \textit{finite} Engel sink, then $g$ has a unique smallest Engel sink. If $\mathscr E(g)$ is the smallest Engel sink of $g$, then the restriction of the mapping $x\mapsto [x,g]$ to $\mathscr E(g)$ must be surjective, which gives the following characterization.

\begin{lemma}[{\cite[Lemma~2.1]{khu-shu20}}]\label{l-min}
If an element $g$ of a group $G$ has a finite Engel sink, then $g$ has a smallest Engel sink $\mathscr E (g)$ and for every $s\in \mathscr E (g)$ there is $k\in {\mathbb N}$ such that  $s=[s,\,{}_kg]$.
\end{lemma}

For profinite groups we proved in \cite{khu-shu20} the following lemma as  a consequence of the Baire Category Theorem~\ref{bct}; we include it here with the proof for the reader's benefit.

\begin{lemma}\label{l-engk}
Suppose that an element $g$ of a profinite group $G$ has a countable Engel sink. Then there are positive integers $i,k$ and a coset $Nb$ of an open normal subgroup $N$  such that
$$
[[nb,\,{}_ig],g^k]=1 \qquad \text{for all}\quad n\in N.
$$
\end{lemma}

\begin{proof} Let $\{s_1,s_2,\dots\}$ be an Engel sink of $g$.
We define the sets
$$
S_{kl}=\{x\in G\mid [x,\,{}_kg]=s_l\}.
$$
Note that each $S_{kl}$ is a closed subset of $G$. Then
$$
G=\bigcup _{k,l}S_{kl}
$$
by the definition of the Engel sink. By Theorem~\ref{bct} one of these sets $S_{ij} $ contains an open subset, which contains  a  coset $Nb$ of an open normal subgroup $N$.
Thus,
$$
[nb,\,{}_ig]=s_j\qquad \text{for all}\quad n\in N.
$$
Since $G/N$ is a finite group, the coset $Nb$ is invariant under conjugation by some power~$g^k$. Then
$$
\begin{aligned}
s_j^{g^k}&=[b,\,{}_ig]^{g^{k}}=[b^{g^{k}},\,{}_ig]\\&=[nb,\,{}_ig]\quad \text{for some } n\in N\\
&=s_j.
\end{aligned}
$$
In other words, $g^{k}$ commutes with $s_j$, so that
\begin{equation*}
[[nb,\,{}_ig],g^{k}]=[s_j,g^k]=1\qquad \text{for all}\quad n\in N.\tag*{\qedhere}
\end{equation*}
\end{proof}

\begin{remark} When a finite group $A$ acts by automorphisms on a group $G$, if $H$ is a subgroup of finite index, then $\bigcap_{a\in A}H^a$ is an  $A$-invariant subgroup of finite index. We  freely use this property throughout the paper without special references.
\end{remark}

Recall that a section of a group $G$ is a quotient $S/T$, where $T$ is a normal subgroup of a subgroup~$S$. If $\f$ is an automorphism (or an element) of $G$, then a section $S/T$ is said to be $\f$-invariant if both $S$ and $T$ are $\f$-invariant subgroups (respectively, normalized by~$\f$).

\begin{remark}
If an element $g$ of a group $G$ has a finite (or countable) Engel sink in~$G$, then its image also has a finite (respectively, countable) Engel sink in any $g$-invariant section of $G$. We  freely use this property throughout the paper without special references.
\end{remark}

\section{Uniformly powerful pro-$p$ groups}

In this section we consider a special class of pro-$p$ groups, the finitely generated uniformly powerful pro-$p$ groups, for which we prove that any element with a countable Engel sink must in fact be an Engel element.

Recall that a pro-$p$ group $G$ is \textit{powerful} if $[G,G]\leq G^p$ for $p\ne 2$, or  $[G,G]\leq G^4$ for $p=2$. Henceforth $G^n$ denotes the (closed) subgroup generated by all $n$-th powers of elements of $G$.
In a powerful pro-$p$ group $G$, for every $k\in \N$  the subgroup $G^{p^k}$ consists of $p^k$-th powers of elements of $G$, and the subgroups $G^{p^k}$ form a central series of $G$.

A powerful pro-$p$ group $G$ is said to be \textit{uniformly powerful} if the mapping $x\mapsto x^p$ induces an isomorphism of  $G^{p^i}/G^{p^{i+1}}$ onto $G^{p^{i+1}}/G^{p^{i+2}}$ for every $i\in \N$. This definition implies that  a uniformly powerful pro-$p$ group is torsion-free.

If $G$ is a finitely generated  uniformly powerful pro-$p$ group, then the structure of a Lie algebra $L$ over the $p$-adic integers $\Z_p$ on the same underlying set $G=L$ can be defined  in a certain canonical way; see \cite[Theorem~4.30]{anal}.
The additive group of $L$ is isomorphic to a direct sum of copies of $\Z_p$.
 The group multiplication in $G$ is then reconstructed from the Lie $\Z_p$-algebra operations in $L$ by the  Baker--Campbell--Hausdorff formula; see  \cite[Theorem~9.10]{anal}.

 Let $X,Y$ be free generators of a free associative algebra of formal power series, which also becomes a Lie algebra with respect to the Lie product $[U,V]_L=UV-VU$. The Baker--Campbell--Hausdorff formula has the form
\begin{equation}\label{e-bch}
\begin{aligned}
\Phi(X,Y)&=\log (\exp(X)\cdot \exp(Y)) \\ &=X+Y+\frac{1}{2}[X,Y]_L+\frac{1}{12}[[X,Y]_L,Y]_L- \frac{1}{12}[[X,Y]_L,X]_L+\cdots ,
\end{aligned}
\end{equation}
where
$$\log (1+Z)=\sum_{i=1}^{\infty}(-1)^i\frac{Z^i}{i}\quad\text{and}\quad \exp(Z)=\sum_{i=0}^{\infty}\frac{Z^i}{i!}$$
are formal power series, and  the right-hand  side of \eqref{e-bch} is a series in Lie algebra commutators of increasing weights in $X,Y$ with rational coefficients. The group operation in the uniformly powerful pro-$p$ group $G$ is expressed in terms of Lie $\Z_p$-algebra operations as
\begin{equation}\label{e-bch2}
xy=\Phi(x,y)=x+y+\frac{1}{2}[x,y]_L+\frac{1}{12}[[x,y]_L,y]_L- \frac{1}{12}[[x,y]_L,x]_L+\cdots ,
\end{equation}
To avoid confusion, in this section we use  $[\,,]_G$ for group commutators in $G$, and $[\,,]_L$ for Lie algebra commutators in $L$, and we write $0_L=1_G$ to denote the zero in $L$ and the identity in $G$. Although the denominators of the coefficients in \eqref{e-bch} may be divisible by powers of $p$, all the terms in \eqref{e-bch2} are  well defined in the $\Z_p$-algebra $L$  because  $[L,L]_L\leq pL$ (or  $[L,L]_L\leq 4L$ when $p=2$). Roughly speaking, if a power $p^k$ appears in the denominator  of a coefficient of a Lie algebra commutator in this series, this commutator always belongs to $p^kL$ and therefore has a unique $p^k$-th root in the additive group of $L$, so that the product is well-defined in $L$. Moreover, the $p$-adic estimates of the coefficients in \eqref{e-bch} show that the condition $[L,L]_L\leq pL$ (or  $[L,L]_L\leq 4L$ when $p=2$) ensures a `surplus' divisibility by powers of $p$ that makes the series convergent in the $p$-adic topology on $L$; see \cite[Theorem~6.28]{anal}. The group commutator is also expressed as a series
$$
[x,y]_G=[x,y]_L+\cdots
$$
in Lie algebra  commutators of increasing weights in $x,y$ with rational coefficients, which is also well defined and convergent.

It is known that \begin{equation}\label{e-iff}
                   [x,y]_G=1\quad\text{if and only if}\quad [x,y]_L=0,
                 \end{equation}
and for $\alpha\in \Z_p$ we have
\begin{equation}\label{e-exp}
  G\ni x^{\alpha}=\alpha x\in L ,
\end{equation}
where   $x^{\alpha}$ is the $\alpha$-th power of an element $x\in G$ and  $\alpha x$ is the  $\alpha $-th multiple  of $x$ regarded as an element of  $L$.

The passages from $G$ to $L$ and vice versa described above constitute a category isomorphism between finitely generated uniformly powerful pro-$p$ groups and uniformly powerful Lie $\Z_p$-algebras; see \cite[Theorem~9.10]{anal}.

The main result of this section is the following.
\begin{proposition}
  \label{pr-unif}
  Let $G$ be a finitely generated uniformly powerful pro-$p$ group. If an element $g\in G$ has a countable Engel sink, then $g$ is an Engel element of $G$.
\end{proposition}

\begin{proof}
Assume that $G$ is in the `category correspondence' with a Lie $\Z_p$-algebra $L$ with the same underlying sets $G=L$, as described above.

By Lemma~\ref{l-engk}, there is  a coset $bN$ of an open normal subgroup $N$ and  positive integers $l,s$ such that   $[[bn,\,{}_lg]_G,\,g^s]_G=1$ for all $n\in N$. By \eqref{e-iff} and \eqref{e-exp} this means that
$$
0_L=[[bn,\,{}_lg]_G,\,g^s]_L=[[bn,\,{}_lg]_G,\,sg]_L=s[[bn,\,{}_lg]_G,\,g]_L.$$
Since $G$ (or the additive group of $L$) is torsion-free, we get $[[bn,\,{}_lg]_G,g]_L=\nobreak 0$, that is, $[bn,\,{}_{l+1}g]_G=1$. To simplify notation, we redenote $l+1$ by $l$, so $[bn,\,{}_{l}g]_G=1$ for all $n\in N$.

We can assume that $N=G^{p^m}$. Then $bN=b+N$ and therefore $b+n\in b+N=bN$ for any $n\in N$. Thus, we have
\begin{equation}\label{e1}
[b+n,\,{}_{l}g]_G=1\quad \text{ for any }n\in N.
\end{equation}
We also have $b^{p^m}\in N$, and the idea is to replace $b$ in \eqref{e1} with powers $b^{1+p^{m}}, b^{1+p^{2m}}, \dots $ belonging to $b+N$  and take  linear combinations of the resulting equations to eliminate~$b$, so  for any given $n\in N$ we will obtain $[n,\,{}_{l}g]_G=1$. For further references, we note that
\begin{equation}\label{e1m}
1=[b+b^{p^{im}}+n,\,{}_{l}g]_G=[(1+{p^{im}})b+n,\,{}_{l}g]_G=0_L.
\end{equation}
for any $n\in N$ and any $i\in \N$.

First we consider a free $\Q$-algebra of formal power series in non-commuting variables  $x,y,z$. This is a Lie $\Q$-algebra with respect to the product $[u,v]_L=uv-vu$. The Baker--Campbell--Hausdorff formula defines the structure of a group on the same set, and elements $x,y,z$ generate a free group; see \cite[Ch.~II]{bou} for details. We denote by $[u,v]_G=[u,v]_L+\cdots $ the group commutator in this group.

We write the group commutator $[x+y,\,{}_{l}z]_G$  as an (infinite) linear combination of Lie algebra commutators with rational coefficients resulting from  repeated application of the Baker--Campbell--Hausdorff formula:
\begin{equation}\label{e-free}
  [x+y,\,{}_{l}z]_G=[x+y,\,{}_{l}z]_L+\cdots.
\end{equation}
Note that when we substitute elements of our uniformly powerful pro-$p$ group $G$ instead of $x,y,z$, the same formula holds with the series on the right becoming convergent in $L$.
Expanding brackets and collecting terms  we represent the right-hand side as
\begin{equation*}
[x+y,\,{}_{l}z]_G=w_0(x,y,z)+w_1(x,y,z)+w_2(x,y,z)+\cdots ,
\end{equation*}
where $w_i(x,y,z)$ is an infinite $\Q$-linear combination of Lie algebra commutators in $x, y, z$ each having weight $i$ in $x$. Substituting $x=b$, $y=n\in N$, $z=a$ we see from \eqref{e1} that
\begin{equation*}
1_G=[b+n,\,{}_{l}g]_G=w_0(b,n,g)+w_1(b,n,g)+w_2(b,n,g)+\cdots  =0_L,
\end{equation*}
where every term $w_i(b,n,g)$ belongs to $L$ as a sum of a convergent series. Moreover, if $p^{f(i)}$ is the least power of $p$ such that $w_i(b,n,g)\in p^{f(i)}L$, then
\begin{equation}\label{e-lim}
  \lim _{i\to\infty} f(i)=\infty .
\end{equation}
This follows from the $p$-adic estimates of the Baker--Campbell--Hausdorff formula and the powerfulness condition $[L,L]\leq pL$ (or $[L,L]\leq 4L$ for $p=2$); see \cite[Theorem~6.28 and Corollary~6.38]{anal}.

Note that \begin{equation}\label{e0}
            w_0(b,n,g)=w_0(0,n,g)=[n,\,{}_{l}g]_G.
          \end{equation}
Our aim is to show that $w_0(b,n,g)=0$ for any $n\in N$, which will imply that $g$ is an Engel element. We fix the elements $b$ and $n\in N$ for what follows.

First we conduct a certain linearization process with the free variables $x,y,z$. We denote for brevity $w_i=w_i(x,y,z)$.  Replacing $x$  in \eqref{e-free} with $x^{1+p^{im}}=(1+p^{im})x$ for $i=1,2,\dots$ we obtain
an infinite family of equations
\begin{equation}\label{e3}
\begin{aligned}
   {[x+y,\,{}_{l}z]_G} &=w_0+w_1+w_2+w_3+\cdots  \\
   [(1+p^{m})x+y,\,{}_{l}z]_G&=w_0+(1+ p^{m})w_1+(1+p^{m})^2w_2+(1+p^{m})^3w_3+\cdots  \\
   [(1+p^{2m})x+y,\,{}_{l}z]_G&=w_0+(1+ p^{2m})w_1+(1+p^{2m})^2w_2+(1+p^{2m})^3w_3+\cdots \\
      \cdots    \cdot \cdot \cdots    \cdots    \cdots    \cdots    \cdots &    \cdots    \cdots    \cdots    \cdots    \cdots   \cdots    \cdots    \cdots    \cdots    \cdots    \cdots    \cdots    \cdots    \cdots    \cdots   \cdots\cdots
\end{aligned}
\end{equation}
The coefficients of the $w_i$ on the right form an infinite matrix of Vandermond type. For every $k=1,2,\dots $ we find a linear combination with rational coefficients $c_0,\dots ,c_{k-1}$ (which may be different for different~$k$) of the top $k$ of these  equations for $i=0,\dots,k-1$   such that its right-hand side would have the form
\begin{equation*}
  w_0+\beta _{k}w_{k}+\beta _{k+1}w_{k+1}+\cdots ,
  \end{equation*}
that is, the coefficients of $w_1,\dots ,w_{k-1}$ would vanish. For that we consider the $k\times k$ matrix $V$ in the left upper corner of the coefficient matrix of the right-hand sides in \eqref{e3}, which is a Vandermond matrix. Then we can take for $c_0,\dots ,c_{k-1}$ the first row of the inverse matrix $V^{-1}$. As a result we obtain
\begin{equation}\label{e4}
  \sum_{i=0}^{k-1}c_i [(1+p^{im})x+y,\,{}_{l}z]_G=w_0+\beta _{k}w_{k}+\beta _{k+1}w_{k+1}+\cdots .
\end{equation}
When we substitute $x=b$, $y=n$, $z=g$ into $[(1+p^{im})x+y,\,{}_{l}z]_G$ we obtain by \eqref{e1m}
$$[(1+p^{im})b+n,\,{}_{l}g]_G=1_G=0_L,$$
but the coefficients $c_i$ have denominators divisible by $p$, so that we cannot use them in $L$. Instead we multiply \eqref{e4} by the determinant $\Delta =\det V$, so that all the products $c_i\Delta $ become integers:
\begin{equation}\label{e5}
  \sum_{i=0}^{k-1}c_i\Delta  [(1+p^{im})x+y,\,{}_{l}z]_G=\Delta  w_0+\beta _{k}\Delta w_{k}+\beta _{k+1}\Delta w_{k+1}+\cdots .
\end{equation}
 Then the substitution makes sense in $L$ for the left-hand side of \eqref{e5}, in which each term becomes
$c_i\Delta  [(1+p^{im})b+n,\,{}_{l}g]_G=1_G=0_L$, for every $i=0,\dots ,k-1$.
However, before we can interpret this substitution in $L$ for the right-hand side of \eqref{e5}, we need to consider the coefficients $\beta _i$. For that we continue using  calculation with the free variables.

Let us say for brevity that a rational number $t/s$, where $t,s\in \Z$, is a \textit{$p$-integer} if the denominator $s$ is coprime to $p$; the $p$-integers are naturally contained in the $p$-adic integers~$\Z_p$.

\begin{lemma}\label{l-beta}
  All the coefficients $\beta _i$ in \eqref{e4} are $p$-integers.
\end{lemma}

\begin{proof}
Recall that the coefficients $\beta_i$ appear after taking the linear combination of the equations~\eqref{e3} with coefficients $c_0,\dots,c_k$ given by the first row of the inverse matrix $V^{-1}$, where
$$
V=\begin{bmatrix}
    1 & 1 & 1 & \cdots & 1 \\[1mm]
    1 & 1+p^{m} & (1+p^{m})^2 & \cdots  & (1+p^{m})^k \\[1mm]
    1 & 1+p^{2m} & (1+p^{2m})^2  & \cdots  & (1+p^{2m})^k \\[1mm]
    \vdots  & \vdots  & \vdots  & \ddots & \vdots \\[1mm]
    1 & 1+p^{km} & (1+p^{km})^2  & \cdots  &(1+p^{km})^k
  \end{bmatrix}.
$$
The inverse matrix $V^{-1}$ is known explicitly (see, for example, \cite[\S\,1.2.3,  Exercise~40]{knu}). The first row  $(c_0,c_1,\dots ,c_{k-1})$ is given by the formulae
$$c_0=\frac{(1+p^m)(1+p^{2m})\cdots (1+p^{(k-1)m})}{(-p^{m})(-p^{2m})\cdots (-p^{(k-1)m})}$$
and
$$c_s=\frac{(1+p^m)(1+p^{2m})\cdots \widehat{ (1+p^{sm})}\cdots (1+p^{(k-1)m})}{p^{sm}(p^{sm}-p^{m})(p^{sm}-p^{2m})\cdots \widehat{(p^{sm}-p^{sm})} \cdots (p^{sm}-p^{(k-1)m})}$$
for $s=1,2,\dots, k-1$, where \ $\widehat{}$ \ means omitting this term.

We already have the $k$ equations with $p$-integer values
$$\beta_0=\sum_{i=0}^{k-1}c_i  =1 $$ and
$$\beta_t=\sum_{i=0}^{k-1}c_i (1+p^{im})^t=0\qquad \text{for}\quad t=1,\dots ,k-1 $$
given by $V^{-1}$.
We need to show that the similar sums with $t\geq k$ are also $p$-integers.
Expanding the powers  $(1+p^{im})^t$ into powers of $p^{im}$ and using the fact that $ \sum_{i=0}^{k-1}c_i  =1$
we see by induction on $t$ that the sums
 $$\beta_t=\sum_{i=0}^{k-1}c_i (1+p^{im})^t$$
 are $p$-integers for $t=1,2,\dots $ if and only if the sums
$$\sum_{i=1}^{k-1}c_i \big(p^{im}\big)^u$$
are $p$-integers for $u=1,2,\dots $ (note that the last sum starts from $i=1$).
We already have this property for $u=1,2,\dots ,k-1$.  It remains to consider the values $u\geq k$. But for these big values of $u$ in fact every product
$$c_i \big(p^{im}\big)^u$$
is a $p$-integer for any $u\geq k$  and every $i=1,2,\dots, k-1$. Indeed, the exponent of the highest power of $p$ dividing the denominator of $c_i$ is  $mi(2k-i-1)/2$, which is less than $mik\leq miu$.
\end{proof}

We return to the proof of Proposition~\ref{pr-unif}. Since all the coefficients $\beta _{k}, \beta _{k+1},\dots $ in \eqref{e5} are $p$-integers by Lemma~\ref{l-beta}, on substitution $x=b$, $y=n$, $z=g$ we obtain an equation in $L$:
 \begin{equation*}
 \begin{aligned}
 0&= \; \sum_{i=0}^{k-1}c_i\Delta   [(1+p^{im})b+n,\,{}_{l}g]_G \\[2mm]
 &=\;\Delta  w_0(b,n,g)+\beta _{k}\Delta w_{k}(b,n,g)+\beta _{k+1}\Delta w_{k+1}(b,n,g)+\cdots ,
  \end{aligned}
\end{equation*}
so that
 \begin{equation*}
\Delta  w_0(b,n,g)= -\beta _{k}\Delta w_{k}(b,n,g)-\beta _{k+1}\Delta w_{k+1}(b,n,g)-\cdots .
\end{equation*}
Since the additive group of $L$ is torsion-free, it follows that
 \begin{equation}\label{e6}
w_0(b,n,g)=-\beta _{k} w_{k}(b,n,g)-\beta _{k+1}w_{k+1}(b,n,g)-\cdots .
\end{equation}
Such an equation holds for every $k=1,2,\dots $ (with different sets of coefficients $\beta _i$ for different $k$). It follows that $w_0(b,n,g)=0$.
 Indeed, recall that $w_i(b,n,g)\in p^{f(i)}L$. Since the coefficients $\beta _{i}$ are $p$-integers, we obtain  that the right-hand side of \eqref{e6} belongs to
 $$p^{\min_{j\geq k}f(j)}L.$$
 Since $\lim _{i\to\infty} f(i)=\infty$ as noted in \eqref{e-lim}, we also have
 $$\lim _{k\to\infty} \min_{j\geq k} f(j)=\infty.$$
 Therefore the validity of \eqref{e6} for all $k\in \N$ implies that  $w_0(b,n,g)\in \bigcap_{i=0}^{\infty} p^iL=0$.

 As a result, we obtain $w_0(b,n,g)=0$, which by \eqref{e0} is equivalent to the desired equation $[n,\,{}_{l}g]_G=1$. This holds for any $n\in N$, and since $G/N$ is a finite $p$-group, it follows that $g$ is an Engel element of $G$.
\end{proof}

\section{Pronilpotent groups}

When $G$ is a pronilpotent group, the conclusion of the main Theorem~\ref{t-main} is equivalent to $G$ being locally nilpotent, and this is what we prove in this section.

\begin{theorem}
\label{t2}
Suppose that $G$ is a pronilpotent group admitting an elementary abelian group of coprime automorphisms $A$ of order $q^2$ for a prime $q$. If for each $a\in A^\#$ every element of the centralizer $C_G(a)$ has a countable Engel sink, then $G$
is locally nilpotent.
\end{theorem}

The proof of Theorem~\ref{t2} is mostly about the case where $G$ is a pro-$p$ group. While the special case of uniformly powerful pro-$p$ groups is all but covered in the previous section, for arbitrary pro-$p$ groups we need different Lie ring methods, which we now describe.

For a prime number $p $, the \textit{Zassenhaus $p $-filtration} of a group $G$ (also called the \textit{$p $-dimension series}) is defined by
$$
G_i=\langle g^{p ^k}\mid g\in \gamma _j(G),\;\, jp ^k\geqslant i\rangle \qquad\text{for}\quad i\in \N .
$$
This is indeed a \textit{filtration} (or an \textit{$N$-series}, or a \textit{strongly central series}) in the sense that
\begin{equation}\label{e-fil}
[G_i,G_j] \leqslant G_{i+j}\qquad \text{for all}\quad i, j.
 \end{equation}

 Then the Lie ring $D_p (G)$ is defined with the additive group
$$
D_p(G)=\bigoplus _{i}G_i/G_{i+1},
$$
where the  factors $Q_i=G_i/G_{i+1}$ are additively written. The Lie product is defined on homogeneous elements $xG_{i+1}\in Q_i$, $yG_{j+1}\in Q_j$ via the group commutators by
$$
[xG_{i+1},\, yG_{j+1}] = [x, y]G_{i+j+1}\in Q_{i+j}
$$
and extended to arbitrary elements of $D_p(G)$ by linearity. Condition~\eqref{e-fil} ensures that this  product is well-defined, and group commutator identities imply that $D_p(G)$ with these operations is a Lie ring. Since all the factors $G_i/G_{i+1}$ have prime exponent~$p $, we can view $D_p(G)$ as a Lie algebra over the field of $p $ elements $\mathbb{F}_p $. We denote  by $L_p (G)$ the subalgebra generated by the first factor $G/G_2$. (Sometimes, the notation $L_p (G)$ is used for $D_p (G)$.) If $u\in G_i\setminus G_{i+1}$, then we define $\delta  (u)=i$ to be the \textit{degree} of  $u$ with respect to the Zassenhaus filtration.

The proofs of the Wilson--Zelmanov theorem \cite[Theorem~5]{wi-ze} on local nilpotency of profinite Engel groups, as well as of Theorems~\ref{t-ass} and \ref{t-20}, were  based on the following deep result of Zelmanov \cite{ze92,ze95,ze17}, which is also used in our paper.

\begin{theorem}[{Zelmanov \cite{ze92,ze95,ze17}}]\label{tz}
 Let $L$ be a Lie algebra over a field and suppose that $L$ satisfies
a polynomial identity. If $L$ can be generated by a finite set $X$ such that every
commutator in elements of $X$ is ad-nilpotent, then $L$ is nilpotent.
\end{theorem}

The following lemma is derived from the Baire Category Theorem~\ref{bct}, the Bahturin--Zaitsev theorem \cite{bah-zai} on Lie algebras with automorphisms whose fixed points satisfy a polynomial identity, and the Wilson--Zelmanov theorem~\cite[Theorem~1]{wi-ze} on Lie algebras of groups satisfying a coset identity.

  \begin{lemma}[{\cite[ Proposition 2.6]{acc-shu16}}] \label{l-pi0}
 Assume that a finite group $F$ acts coprimely on a profinite group~$G$ in such a manner that $C_G(F)$ is locally nilpotent. Then for each prime~$p$ the Lie algebra $L_p(G)$ satisfies a multilinear polynomial identity.
  \end{lemma}

We now consider the case of pro-$p$ groups.

\begin{proposition}
\label{pr-pro-p}
Suppose that $P$ is a finitely generated pro-$p$ group admitting an elementary abelian group of coprime automorphisms $A$ of order $q^2$  for a prime $q$. If for each $a\in A^\#$ every element of the centralizer $C_G(a)$ has a countable Engel sink, then $P$ is nilpotent.
\end{proposition}

The plan of the proof of this proposition is as follows. First we prove that the Lie algebra $L_p(P)$ is nilpotent (which is actually proved for the Lie algebra obtained by extension of the ground field). This will mean that the group $P$ is $p$-adic analytic and therefore has a characteristic  subgroup that is a uniformly powerful pro-$p$ group. By Proposition~\ref{pr-unif}, an element with a countable Engel sink in a uniformly powerful pro-$p$ group is actually an Engel element. Then the application of Theorem~\ref{t-ass} gives nilpotency of the uniformly powerful subgroup and therefore the solubility of $P$. The proof is completed by induction on the derived length.

We begin with ad-nilpotency of homogeneous elements of $C_{L_p(P)}(a)$ for $a\in A^\#$.

\begin{lemma}\label{l-ad}
For each $a\in A^\#$ every homogeneous element $\bar g$ of $L_p(P)$ that belongs to $C_{L_p(P)}(a)$ is ad-nilpotent.
\end{lemma}

\begin{proof}
By Lemma~\ref{l-cover} the element $\bar g$ is the image of some element $g\in C_P(a)$  in the corresponding factor  $P_{\delta (g)}/P_{\delta (g)+1}$ of the Zassenhaus filtration, where $\delta (g)$ is the degree of~$g$. We fix the notation $g$ and $\bar g$ for the rest of the proof of this lemma.

By Lemma~\ref{l-engk} for any element $g\in C_P(a)$ there are positive integers $i,s$ and  a coset $Nb$ of an open normal subgroup $N$ such that
$$
[[nb,\,{}_ig],g^{s}]=1\qquad \text{for all}\quad n\in N.
$$
Since $P$ is a pro-$p$ group, we can assume that $s$ is a power of $p$, so that
\begin{equation}\label{eq-ad}
[[nb,\,{}_ig],g^{p^k}]=1\qquad \text{for all}\quad n\in N.
\end{equation}

For generators $x,y,z,t$ of a free group, write
$$
[[xy,\,{}_i z],t]=[[x,\,{}_i z],t]\cdot [[y,\,{}_i z],t]\cdot v(x,y,z,t),
$$
where the word $v(x,y,z,t)$ is a product of commutators of weight at least $i + 3$, each of which involves $x$, $y$, $t$ and involves $z$ at least $i$ times. Substituting $x=n$, $y=b$, $z=g$, and $t=g^{p^k}$ and using \eqref{eq-ad} we obtain that
$$
[[n,\,{}_ig],g^{p^k}]=v(n,b,g,g^{p^k})^{-1}
\qquad \text{for all}\;\, n\in N.
$$
If $|P/N|=p^m$, then for any $h\in P$ we have $[h,\,{}_mg]\in N$, so that we also have
\begin{equation}\label{eq-ad3}
[[h,\,{}_{i+m}g],g^{p^k}]=v([h,\,{}_mg],b,g,g^{p^k})^{-1}.
\end{equation}
We claim that $\bar g$ is ad-nilpotent in $L_p(P)$ of index $i+m+p^k$.

Recall that $\delta  (u)$ denotes the degree of an element $u \in P$ with respect to the Zassenhaus filtration.  It is known that
\begin{equation}\label{e-pd}
u^p\in P_{p\delta (u)}.
\end{equation}
  Furthermore, in $L_p(P)$ for the images of $u$ and $u^p$ in  $P_{\delta (u)}/P_{\delta (u)+1}$ and $P_{p\delta (u)}/P_{p\delta (u)+1}$, respectively, we have
\begin{equation}\label{e-pd2}
[x, \bar{u^p}]=[x,\,{}_p \bar{u}]
\end{equation}
(see, for example, \cite[Ch.~II, \S\,5, Exercise~10]{bou}).
By  \eqref{e-pd}  the degree of $v([h,\,{}_ma],b,g,g^{p^k})$ on the right of~\eqref{eq-ad3} is at least $\delta (b)+\delta (h)+(i+m+p^k)\delta (g)$,
   strictly greater than $d=\delta (h)+(i+m+p^k)\delta (g)$. This means that the image of the right-hand side of \eqref{eq-ad3}  in $P_d/P_{d+1}$ is trivial. At the same time, by \eqref{e-pd2} the image of the left-hand side of \eqref{eq-ad3}  in $P_d/P_{d+1}$ is equal to the image of $[h,\,{}_{i+m+p^k} g]$ in $P_d/P_{d+1}$, which is in turn equal to the  element    $[\bar h,\,{}_{i+m+p^k}\bar g]$ in $L_p(P)$.
 Thus, for the corresponding homogeneous elements of $L_p(P)$ we have
 $$
 [\bar h,\,{}_{i+m+p^k}\bar g]=0.
$$
 Since here $\bar h$ can be any homogeneous element, we obtain that $\bar g$ is ad-nilpotent of index $i+m+p^k$, as claimed.
\end{proof}

The Lie algebra $L_p(P)$ is generated by its first homogeneous component $L_1=P/(P^p[P,P])$, which is the additively written Frattini quotient of $P$. By Lemmas~\ref{l-q2} and  \ref{l-cover} we have
$$L_1=\sum_{a\in A^\#}C_{L_1}(a).$$

Thus we obtain a finite set of generators of $L_p(P)$ that are ad-nilpotent by Lemma~\ref{l-ad}. But we also need all commutators in the generators to be ad-nilpotent. We cannot, however, say that a commutator  in elements of $C_{L_1}(a)$ for different $a\in A^\#$ is again an element of $C_{L_p(P)}(b)$ for some $b\in A^\#$. Instead, we extend the ground field by a primitive $q$-th  of unity $\zeta$ by forming $\tilde L=L_p(P)\otimes_{\Z}\Z[\zeta ]$ and choose a generating set of $\tilde L $ consisting of common eigenvectors for $A$.  We shall prove that $\tilde L$ is nilpotent using Theorem~\ref{tz}. This will obviously imply the nilpotency of $L_p(P)$.

The next two lemmas confirm that the hypotheses in Theorem~\ref{tz} are satisfied for $\tilde L$.

\begin{lemma}\label{l-ad2}
The Lie algebra $\tilde L$ is generated by finitely many elements all commutators in which are ad-nilpotent.
\end{lemma}

\begin{proof}
The Lie algebra $\tilde L$ is generated by its first homogeneous component
$\tilde L_1=L_1\otimes_{\Z}\Z[\zeta ]$, which is a finite $p$-group. Since the ground field of $\tilde L$ is a splitting field for the linear transformations induced by $A$ and $A$ is abelian of coprime order, every $L_i$ is a sum of common eigenspaces for $A$.  In particular, we can choose finitely many generators of  $\tilde L$ among common eigenvectors for $A$ in $ \tilde L_1$. Note that since $A$ is non-cyclic, every common eigenspace for $A$ is contained in one of the centralizers $C_{\tilde L}(a)$ for $a\in A^\#$.

Therefore, since a commutator in common eigenvectors for $A$ is again a common eigenvector for $A$, every commutator of weight $i$ in these generators belongs to $C_{\tilde L_i}(a)$ for some  $a\in A^\#$. It remains to prove that any  element  $u\in C_{\tilde L_i}(a)$ is ad-nilpotent, for each $a\in A^\#$ and for any $i$.

Clearly, $C_{\tilde L_i}(a)=C_{L_i}(a)\otimes_{\Z}\Z[\zeta ]$. Hence,
$$u=v_0+\zeta v_1+\zeta ^2v_2+\cdots +\zeta ^{q-2}v_{q-2}$$
for some $v_i\in C_{L_i}(a)$. Each of the summands $\zeta ^iv_i$ is ad-nilpotent by Lemma~\ref{l-ad}. A~sum of ad-nilpotent elements need not be
ad-nilpotent in general. But in our case, we shall see that these summands generate a
nilpotent subalgebra, and we shall be able to apply the following lemma.

\begin{lemma}[{\cite[Lemma~5]{khu-shu99}}]\label{l-l5}
  Suppose that $M$ is a Lie algebra, $H$ is a subalgebra of $M$
generated by $s$ elements $h_1, \dots , h_s$ such that all commutators in the $h_i$ are ad-nilpotent of index $t$. If $H$ is nilpotent of class $c$, then for some $(s, t, c)$-bounded number $\e $, we have
$$[M, \underbrace{H,\dots  , H}_\e]=0.$$
\end{lemma}

Recall that we are proving that $u=v_0+\zeta v_1+\zeta ^2v_2+\cdots +\zeta ^{q-2}v_{q-2}$ is ad-nilpotent, where $v_i\in C_{L_i}(a)$.
Let $H$ be the Lie subalgebra generated by the elements $v_0, \zeta v_1, \zeta ^2v_2, \dots , \zeta ^{q-2}v_{q-2}$. It is contained in $C_{\tilde L}(a)=C_{L_p(P)}(a)\otimes_{\Z}\Z[\zeta ]=\bigoplus C_{L_i}(a)\otimes_{\Z}\Z[\zeta ]$. By Lemma~\ref{l-cover} each $C_{L_i}(a)$ is the image of a subgroup  of $C_{P}(a)$ in the corresponding factor of the Zassenhaus $p$-filtration. All elements of the pro-$p$ group $C_{P}(a)$ have countable Engel sinks by hypothesis, and therefore $C_{P}(a)$ is locally nilpotent by Theorem~\ref{t-20}.  It follows that
$C_{L_p(P)}(a)$ and $C_{\tilde L}(a)$ are also locally nilpotent. In particular, the Lie subalgebra $H=\langle v_0, \zeta v_1,  \dots , \zeta ^{q-2}v_{q-2}\rangle$ is nilpotent of certain class $c$. There are only finitely many commutators of weight at most $c$ in the generators $v_0, \zeta v_1,  \dots , \zeta ^{q-2}v_{q-2}$, each of which has the form $\zeta ^k w$, where $w$ is a homogeneous element of $L_p(P)$ contained in $C_{L_p(P)}(a)$. By Lemma~\ref{l-ad} all such commutators are ad-nilpotent, so taking $t$ to be the maximum of their ad-nilpotency indices we can apply Lemma~\ref{l-l5} to obtain that $u=v_0+\zeta v_1+\zeta ^2v_2+\cdots +\zeta ^{q-2}v_{q-2}$ is also ad-nilpotent.
\end{proof}

\begin{lemma}\label{l-pi}
  The Lie algebra $\tilde L$ satisfies a polynomial identity.
  \end{lemma}

\begin{proof}
  Choose any $a\in A^\#$. The centralizer $C_{P}(a)$ is locally nilpotent by Theorem~\ref{t-ass}.  Then $L_p(P)$ satisfies a multilinear polynomial identity by Lemma~\ref{l-pi0}. Since the Lie algebra $L_p(G)$ satisfies a multilinear polynomial identity, the same identity holds on $\tilde L$.
\end{proof}

\begin{proof}[Proof of Proposition~\ref{pr-pro-p}] Recall that  $P$ is a finitely generated pro-$p$ group admitting an elementary abelian group of coprime automorphisms $A$ of order $q^2$ for a prime $q$ such that for each $a\in A^\#$ every element of the centralizer $C_G(a)$ has a countable Engel sink; we need to show that $P$ is nilpotent. By Lemmas~\ref{l-pi} and \ref{l-ad2} the Lie algebra $\tilde L=L_p(P)\otimes_\Z\Z[\zeta ]$ satisfies the hypotheses of Theorem~\ref{tz}, by which $\tilde L$ is nilpotent, and therefore the Lie algebra $L_p(P)$ is also nilpotent.
The nilpotency of $L_p (P)$  for the finitely generated pro-$p $ group $P$  is equivalent to $P$ being a $p$-adic analytic group. This result goes back to Lazard~\cite{laz}; see also \cite[Proposition~D]{sha}. In terms of the Lubotzky--Mann theory of powerful pro-$p$ groups \cite{lub-man2}, another result of Lazard \cite[III, 3.1.3, 3.4.4]{laz} states that a finitely generated pro-$p$ group  $P$
is $p$-adic analytic if and only if it has a powerful subgroup of finite index, and therefore $P$ has a characteristic open subgroup $U$, which is a uniformly powerful pro-$p$ group of finite rank (see \cite[Corollaries 4.3 and 8.34]{anal}).

By Proposition~\ref{pr-unif}, elements with countable Engel sinks in a uniformly powerful pro-$p$ group are actually Engel elements. Therefore $A$ acts on $U$ in  such a way that for each $a\in A^\#$ every element of the centralizer $C_U(a)$ is an Engel element. By Theorem~\ref{t-ass} the group $U$ is nilpotent. As a result, the group $P$ is soluble.

We now complete the proof by induction on the derived length of $P$.
By induction hypothesis, $P$ has an abelian characteristic subgroup $V$ such that $P/V$ is  nilpotent. By Theorem~\ref{t-ass} it is sufficient to show that  for each $a\in A^\#$ every element of the centralizer $C_P(a)$ is an Engel element.   Since $P/V$ is nilpotent, it is sufficient to show that every element $g\in C_P(a)$ is an Engel element in the product $V\langle g\rangle$. Since all elements of $V\langle g\rangle$ have countable Engel sinks, the group $V\langle g\rangle$ is nilpotent by Theorem~\ref{t-20}, whence the result.\end{proof}

\begin{proof}[Proof of Theorem~\ref{t2}]
By
Theorem~\ref{t-ass},  it is sufficient to prove that for each $a\in A^\#$ every element $h\in C_G(a)$ is an Engel element in $G$.
For each prime~$p$, let $G_p$ denote the Sylow $p$-subgroup of $G$, so that $G$ is a Cartesian product of the~$G_p$, since $G$ is pronilpotent. Given any two elements $g\in G$ and $h\in C_G(a)$, we write $g=\prod _pg_p$ and $h=\prod _ph_p$,  where $g_p,h_p\in G_p$. Clearly, $[g_q,h_p]=1$ for $q\ne p$.

By Lemma~\ref{l-engk}, for the element $h\in C_G(a)$ there are positive integers $i,k$ and a coset $Nb$ of an open normal subgroup $N$  such that
\begin{equation}\label{e-eng2}
[[nb,\,{}_ih],h^{k}]=1\qquad \text{for all}\quad n\in N.
\end{equation}
Let $l$ be the (finite) index of $N$ in $G$. Then $N$ contains all Sylow $q$-subgroups of $G$ for $q\not\in \pi (l)$. Hence we can choose $b$ to be a $\pi(l)$-element. Let $\pi=\pi (l)\cup \pi (k)$; note that $\pi$ is a finite set of primes.

We claim that
$$
[g_q,\,{}_{i+1}h_q]=1\qquad \text{for }q\not\in \pi.
$$
 Indeed, since $b$ commutes with elements of $G_q$ for $q\not\in \pi$ and $G_q\leq N$,  by \eqref{e-eng2} we have
\begin{equation}\label{eq-engq2}
\begin{aligned}
 1=[g_qb,\,{}_ih],h^{k}]
 & = [[g_q,\,{}_ih],h^{k}]\cdot  [[b,\,{}_ih],h^{k}]\\
 & =[[g_q,\,{}_ih],h^{k}]\\
 &= [[g_q,\,{}_ih_q],h_q^{k}].
 \end{aligned}
\end{equation}
 Thus, $h_q^{k}$ centralizes $[g_q,\,{}_ih_q]$. Since $k$ is coprime to $q$, we have $ \langle h_q^{k}\rangle =\langle h_q\rangle$. Therefore \eqref{eq-engq2} implies that $[[g_q,\,{}_ih_q],h_q]=1$, as claimed.

 For every prime $p$ the group $G_p$ is locally nilpotent by Proposition~\ref{pr-pro-p},  so there is $k_p$ such that $[g_p,\,{}_{k_p}h_p]=1$.  Now for $m=\max\{i+1, \max_{p\in \pi} \{k_p\}\}$ we have $[g_p,\,{}_{m}h_p]=1$ for all $p$, which means that  $[g,\,{}_{m}h]=1$.
 Thus, for each $a\in A^\#$ every element  $h\in C_G(a)$ is an Engel element in $G$. Therefore $G$ is locally nilpotent by
 Theorem~\ref{t-ass}.
  \end{proof}

\section{Profinite groups}

First we recall some lemmas in \cite{khu-shu20} about Engel sinks of coprime automorphisms.

 \begin{lemma}[{\cite[Lemma~4.7]{khu-shu20}}]\label{l-copr2}
 Let $\varphi$ be a coprime automorphism of a pronilpotent
group~$F$ with a countable Engel sink $\mathscr E(\varphi)$ in the semidirect product $F\langle \varphi\rangle$. Then the set $K=\{ [g,\varphi]\mid g\in F\}$ is a finite smallest Engel sink of $\varphi$ in the semidirect product $F\langle \varphi\rangle$.
 \end{lemma}

The next lemma was basically proved in \cite[Lemma~4.7]{khu-shu20}, but we cannot make a direct reference and therefore give a short proof.

 \begin{lemma}\label{l-copr3}
 Let $\varphi$ be a coprime automorphism of a  locally nilpotent
 profinite group~$F$. If  $\f$ has a countable Engel sink $\mathscr E(\varphi)$ in the semidirect product $F\langle \varphi\rangle$, then $\gamma _{\infty} (F\langle \varphi\rangle)$ is finite  and  $\gamma _{\infty}(F\langle \varphi\rangle)= [F,\varphi]$.
 \end{lemma}

 \begin{proof}
By Lemma~\ref{l-copr2}, the set $K=\{ [g,\varphi]\mid g\in F\}$ is finite. Then the commutator subgroup $[F,\varphi ]=\langle K\rangle$  is nilpotent, since $F$ is locally nilpotent. By Lemma~\ref{l-gff},
\begin{equation}\label{e-ff}
[[F,\varphi],\varphi]=[F,\varphi].
\end{equation}

Let $V$ be the quotient of  $[F,\varphi ]$ by its derived subgroup. For any $u,v\in V$ we have $[uv,\varphi ]=[u,\varphi][v,\varphi]$, since $V$ is abelian, and $[V,\varphi ]=V$ by \eqref{e-ff}. Hence $V$ consists of the images of elements of $K$, and therefore is finite. Then the nilpotent group $[F,\varphi]$ is also finite (see, for example, \cite[5.2.6]{rob}).

The quotient $F\langle \varphi\rangle/[F,\varphi]$ is  the direct product of the images of $F$ and $\langle \varphi\rangle$ and therefore is pronilpotent. Hence,  $\gamma _{\infty}(F\langle \varphi\rangle)\leq [F,\varphi]$ and in view of
 \eqref{e-ff} we obtain
$\gamma _{\infty}(F\langle \varphi\rangle)= [F,\varphi]$. In particular, $\gamma _{\infty}(G\langle \varphi\rangle)$ is finite.
 \end{proof}

\begin{proof}[Proof of Theorem~\ref{t-main}]
We now embark on the proof of the main result.
Recall that $G$ is a profinite group admitting a coprime elementary abelian group of automorphisms $A$ of order $q^2$ for a prime $q$ such that  for each $a\in A^\#$ all elements in $C_G(a)$ have countable Engel sinks. We want to prove that $G$ is finite-by-(locally nilpotent).

By Theorem~\ref{t2} any pronilpotent $A$-invariant section of $G$ is locally nilpotent. In particular, every Sylow $p$-subgroup of $G$ is locally nilpotent, since there is an $A$-invariant Sylow $p$-subgroup. The next lemma extends Lemma~\ref{l-copr3}.

\begin{lemma}\label{l-511}
Suppose that a section $S=FC$ of the group $G$ is a product of a normal locally nilpotent subgroup $F$ and a subgroup $C$ such that every element of $C$ has a countable Engel sink in $S$. Then
$\g _{\infty}(F\langle g\rangle)$ is finite for every $g\in S$.
\end{lemma}

\begin{proof} Write $g=fh$ for $f\in F$ and $h\in C$.
 Since  $F\langle g\rangle=F\langle h\rangle$, we work with $h\in C$.
For a prime $p$, let $P$ be a Sylow $p$-subgroup of $F(G)$, and write $h=h_ph_{p'}$, where $h_p$ is a $p$-element and $h_{p'}$ a $p'$-element both lying in $C$, and $[h_p,h_{p'}]=1$. Then $P\langle h_p\rangle$ is a normal Sylow $p$-subgroup of  $P\langle h\rangle$, on which $h_{p'}$ induces by conjugation a coprime automorphism. By Lemma~\ref{l-copr3} the subgroup  $[P,h_{p'}]=[P\langle h_p\rangle,h_{p'}]=\g _{\infty}(P\langle h\rangle)$ is finite.

The quotient  $P\langle h\rangle/\g _{\infty}(P\langle h\rangle)$ is the direct product of the images of $\langle h_{p'}\rangle$ and $P\langle h_p\rangle$. The latter is a  pro-$p$ group, which is locally nilpotent by Theorem~\ref{t2} as noted above.
  Thus, the quotient  $P\langle h\rangle/\g _{\infty}(P\langle h\rangle)$ is locally nilpotent, and therefore we can choose a finite smallest Engel sink $\mathscr E_p(h)\subseteq \gamma _{\infty}(P\langle h\rangle)$ of $h$ in $P\langle h\rangle$.

Note that
\begin{equation}\label{e-equiv}
   \text{if}\quad \mathscr E_p(h)=\{ 1\}, \quad\text{then}\quad \gamma _{\infty}(P\langle h\rangle)=1.
\end{equation}
Indeed, if  $\mathscr E_p(h)=\{ 1\}$, then, in particular, the image $\bar h$ of $h$ in $\langle h\rangle/C_{\langle h\rangle}([P,h_{p'}])$ is an Engel element of the finite group $[P,h_{p'}]\langle \bar h\rangle$ and therefore $\bar h$ is contained in its Fitting subgroup by  Baer's theorem  \cite[12.3.7]{rob}. Then
$$
\gamma _{\infty}(P\langle h\rangle)=[P,h_{p'}]=[[P,h_{p'}],h_{p'}]=[[P,h_{p'}],\bar h_{p'}]=1.
$$

By Lemma~\ref{l-min}, for every $s\in \mathscr E_p(h)$ we have $s=[s,\,{}_kh]$ for some $k\in {\mathbb N}$, and then also
\begin{equation}\label{e-cycl}
s=[s,\,{}_{kl}h]\quad \text{for any}\;\, l\in {\mathbb N}.
\end{equation}

We claim that $\mathscr E_p(h)=\{1\}$ for all but finitely many primes $p$. Suppose the opposite, and $\mathscr E_{p_i}(h)\ne \{1\}$ for each prime $p_i$ in an infinite set of primes~$\pi$. Choose a nontrivial element $s_{p_i}\in \mathscr E_{p_i}(h)$ for every $p_i\in \pi$. For any subset $\sigma\subseteq \pi$, consider the (infinite) product
$$
s_{\sigma}=\prod _{p_j\in \sigma} s_{p_j}.
$$
Note that the elements $s_{p_j}$ commute with one another  belonging to different normal Sylow subgroups of $F$.  If  $\mathscr E(h)$ is any Engel sink of $h$ in $H$, then for some $k\in {\mathbb N}$ the commutator $[s_{\sigma},\,{}_kh]$ belongs to $\mathscr E(h)$. Because of the properties \eqref{e-cycl}, all the components of $[s_{\sigma},\,{}_kh]$ in the Sylow $p_j$-subgroups of $F$ for $p_j\in \sigma$ are non-trivial, while all the other components in Sylow $q$-subgroups for $q\not\in \sigma$ are trivial by construction. Therefore for different subsets $\sigma\subseteq \pi$ we thus obtain different elements of  $\mathscr E(h)$. The infinite set $\pi$ has continuum of different subsets, whence $\mathscr E(h)$ is uncountable, contrary to $h$ having a countable Engel sink by the hypothesis.

Thus, for all but finitely many primes $p$ we have $\mathscr E_p(h)=\{1\}$, which is the same as $\gamma _{\infty}(P\langle h\rangle)=1$ by \eqref{e-equiv}. Therefore the subgroup
$$
\gamma _{\infty}(F\langle g\rangle)=\gamma _{\infty}(F\langle h\rangle)=\prod_p\gamma _{\infty}(P\langle h\rangle)
$$
 is finite.
\end{proof}

Our next step in the proof of Theorem~\ref{t-main} is proving that $G$ has an open locally nilpotent subgroup. Recall that by Theorem~\ref{t2} any pronilpotent $A$-invariant section of $G$ is locally nilpotent. In particular, the largest pronilpotent normal subgroup of an $A$-invariant section $H$ is also the largest locally nilpotent normal subgroup of $H$, and we call it the Fitting subgroup $F(H)$ of $H$. Then further terms of the Fitting series are defined as usual by induction: $F_1(H)=F(H)$ and $F_{i+1}(H)$ is the inverse image of $F(H/F_i(H))$. A group $H$ has finite Fitting height $k$ if $F_k(H)=H$ for $k\in {\mathbb N}$ and $k$ is the least number with this property.

\begin{proposition}\label{pr-virt}
The group $G$ has an open $A$-invariant normal locally nilpotent subgroup.
\end{proposition}

\begin{proof}
For each $a\in A^\#$ the centralizer $C_G(a)$ is finite-by-(locally nilpotent)  by Theorem~\ref{t-20}. Hence $G$ has an $A$-invariant open subgroup $H$ such that each centralizer $C_H(a)$ for $a\in A^\#$ is locally nilpotent. By Lemma~\ref{l-cover} the same holds for every finite quotient of $H$ by an $A$-invariant open normal  subgroup.
By a theorem of Ward \cite{war73} such finite quotients are soluble and have Fitting height at most 2. Hence $H$ is prosoluble
and $h(H)\leq 2$.

\begin{lemma}\label{l-51}
  For each $a\in A^\#$ the subgroup $F(H)C_H(a)$ is finite-by-(locally nilpotent).
\end{lemma}

\begin{proof}
By Theorem~\ref{t-20}, it is sufficient to show that every element  $g\in F(H)C_H(a)$ has a finite Engel sink.  Since $C_H(a)$ is locally nilpotent, an Engel sink of $g$ in $F(H)\langle g\rangle$ is also an Engel sink of $g$ in $F(H)C_H(a)$.
By Lemma~\ref{l-511} the subgroup $\g _{\infty}(F(H)\langle g\rangle)$ is finite. If the pronilpotent quotient $F(H)\langle g\rangle/\g _{\infty}(F(H)\langle g\rangle)$ was locally nilpotent, we would obtain a finite Engel sink of $g$ in $\g _{\infty}(F(H)\langle g\rangle)$. But we cannot immediately apply Theorem~\ref{t2}, since this quotient is not $A$-invariant.

To work around this difficulty we need to consider $A$-invariant subgroups.  Write $g=fh$ for $f\in F(H)$ and $h\in C_H(a)$. Let
$$M=F(H)\langle g^A\rangle=F(H)\langle h^A\rangle,$$
where, as usual, $\langle x^A\rangle=\langle x^a\mid a\in A\rangle$. Since  $F(M)\geq F(H)$, it is sufficient to show that $g$ has a finite Engel sink in $M$. Since $M\leq H$, we have $M=F(M)C_M(a)$, the centralizer $C_M(a)$ is locally nilpotent, and by Lemma~\ref{l-cover} the same holds for any $A$-invariant section of $M$.

Since $C_H(a)$ is locally nilpotent, the subgroup $\langle h^ A\rangle$ is nilpotent, and therefore $M/F(M)$ is a finitely generated nilpotent group. We now construct by induction a finite series of nested finite normal subgroups $K_i$ of $M$ such that the nilpotency class of $(M/K_i)/F(M/K_i)$ diminishes at every step, up to class $0$ corresponding to the trivial group. As a result we will obtain a finite normal subgroup of $M$ with locally nilpotent quotient, thus proving that $g$ has a finite Engel sink in $M$.

As a basis of the inductive construction we put $K_0=1$. Suppose that we have already constructed an $A$-invariant finite normal subgroup $K_i$ such that  the nilpotency class of $\bar M/F(\bar M)$ is $c\geq 1$, where bars denote images in  $\bar M=M/K_i$.

Choose finitely many elements $z_1,\dots ,z_m$ whose images generate the centre of $\bar M/F(\bar M)$. This is possible, since $\bar M/F(\bar M)$ is a finitely generated nilpotent group.  Then each subgroup $F(\bar M)\langle z_i^a\rangle$ for $a\in A$ is normal in $\bar M$. The subgroups $\g _{\infty}(F(\bar M)\langle z_i^a\rangle)$ are also normal in $\bar M$.
By Lemma~\ref{l-511}
the subgroups $\g _{\infty}(F(\bar M)\langle z_i^a\rangle)$ are finite.
Therefore their product is a finite normal $A$-invariant subgroup
$$K=\prod _{i=1}^m\prod_{a\in A} \g _{\infty}(F(\bar M)\langle z_i^a\rangle)
,$$
and we define $K_{i+1}$ to be the  full inverse image of $K$.
Consider  the product
$$L=\prod _{i=1}^m\prod_{a\in A} F(\bar M)\langle z_i^a\rangle. $$
Its image $L/K$ in the quotient $\bar M/K$ is a product of  the images of the normal pronilpotent subgroups  $F(\bar M)\langle z_i^a\rangle / \g _{\infty}(F(\bar M)\langle z_i^a\rangle)$ and is therefore  pronilpotent. By Theorem~\ref{t2} the quotient $L/K$ is locally nilpotent. Hence the nilpotency class of $\bar M/F(\bar M)$ is less than $c$.

The above construction terminates when $M/K_k$ becomes locally nilpotent, at some finite step $k$. Thus, $M$  has a finite normal $A$-invariant subgroup $N=K_k$ with locally nilpotent quotient, and therefore the element $g$ has a finite Engel sink in $M$ contained in $N$. As a result, every element  of $F(H)C_H(a)$ has a finite Engel sink and therefore $F(H)C_H(a)$ is finite-by-(locally nilpotent) by Theorem~\ref{t-20}.
\end{proof}

We return to the proof of Proposition~\ref{pr-virt}. Since every subgroup $F(H)C_H(a)$ is finite-by-(locally nilpotent) by Lemma~\ref{l-51}, there is an open normal $A$-invariant subgroup $R$ such that $R\cap \gamma_\infty(F(H)C_H(a))=1$ for each $a\in A^\#$. Since $F(R)\leq F(H)$ and $C_R(a)\leq C_H(a)$, we have
$$
 \gamma_\infty(F(R)C_R(a))\leq \gamma_\infty(F(H)C_H(a))\cap R=1
 $$
for each $a\in A^\#$. By Theorem~\ref{t2} we obtain that  for each $a\in A^\#$  the subgroup $F(R)C_R(a)$ is locally nilpotent. Since $R/F(R)$ is also locally nilpotent, it follows that every element in $C_R(a)$ is Engel in $R$, for each $a\in A^\#$. Then $R$ is locally nilpotent by Theorem~\ref{t-ass}.
  \end{proof}

We now complete the proof of Theorem~\ref{t-main}. Since any pronilpotent $A$-invariant section of $G$ is locally nilpotent by Theorem~\ref{t2}, we only need to prove that $\g_{\infty}(G)$ is finite.  By Proposition~\ref{pr-virt} the quotient $G/F(G)$ by the Fitting subgroup is finite. We proceed by induction on the index of $F(G)$. If $G/F(G)$ has a proper normal $A$-invariant subgroup $B/F(G)$, then by the induction hypothesis $\g_{\infty}(B)$ is finite, and it remains to apply the induction hypothesis to $G/\g_{\infty}(B)$.

Thus, we can assume that $G/F(G)$ has no proper normal $A$-invariant subgroups.
First consider the case where $G/F(G)$ is abelian. Then $G=F(G)C_G(a)$ for some $a\in A^\#$. For every $g\in G$ the subgroup  $\g _{\infty}(F(G)\langle g\rangle)$ is finite by Lemma~\ref{l-511} and it is a normal subgroup of $F(G)$. Since $G/F(G)$ is finite, $\g _{\infty}(F(G)\langle g\rangle)$ has finitely many conjugates in $GA$ and their product $E$ is a finite normal $A$-invariant subgroup. The $A$-invariant  section
$F(G)\langle g^A\rangle/E$ is a product of the images of the normal pronilpotent subgroups $F\langle g^a\rangle/ \g _{\infty}(F\langle g^a\rangle)$ and therefore is pronilpotent and locally nilpotent by Theorem~\ref{t2}. Hence $g$ has a finite Engel sink contained in $E$. Thus, every element of $G$ has a finite Engel sink, and therefore $G$ is finite-by-(locally nilpotent) by Theorem~\ref{t-20}.

We now consider the case where $G/F(G)$ is a direct product of non-abelian finite simple groups. Let $p$ be a prime divisor of $|G/F(G)|$, and let $P$ be an $A$-invariant Sylow $p$-subgroup of $G$. By the induction hypothesis, $\g _{\infty}(F(G)P)$ is finite, and it is a normal $A$-invariant subgroup of $F(G)$. Since $G/F(G)$ is finite, $\g _{\infty}(F(G)P)$ has finitely many conjugates in $GA$ and their product is a finite normal $A$-invariant subgroup. Passing to the quotient by this subgroup we can assume that $\g _{\infty}(F(G)P)=1$. This means that $[P,F_{p'}]=1$, where $F_{p'}$ is the Hall ${p'}$-subgroup of $F(G)$. Then also $[P^G,F_{p'}]=1$, where $P^G$ is the normal closure of $P$. Since $G=P^GF(G)$, it follows that $\g _{\infty}(G)=\g _{\infty}(P^G)$ and therefore $[\g _{\infty}(G), F_{p'}]=1$.

We repeat  this argument with another prime divisor $q$ of $|G/F(G)|$ and an $A$-invariant Sylow $q$-subgroup $Q$ of $G$, in addition assuming that $\g _{\infty}(F(G)Q)=1$. We obtain that  $\g _{\infty}(G)=\g _{\infty}(Q^G)$ and $[\g _{\infty}(G), F_{q'}]=1$.

Since $F(G)=F_{p'}F_{q'}$, it follows that $[\g _{\infty}(G), F(G)]=1$.

We also have
$$\g _{\infty}(G)=[\g _{\infty}(G),G]=[\g _{\infty}(G), \g _{\infty}(G)F(G)]=[\g _{\infty}(G),\g _{\infty}(G)].$$

As a result, $\g _{\infty}(G)\cap F(G)$ is contained in the centre of $\g _{\infty}(G)$. Since $\g _{\infty}(G)F(G)=G$, this means that $\g _{\infty}(G)\cap F(G)$ is isomorphic to a homomorphic image of the Schur multiplier of the finite group $G/F(G)$. Since the Schur multiplier of a finite group is finite \cite[Hauptsatz~V.23.5]{hup}, we obtain that $\gamma _{\infty}(G)\cap F(G)$ is finite, and so is $\gamma _{\infty}(G)$, as required.
\end{proof}

 \section*{Acknowledgements}
The first author was supported by the Mathematical Center in Akademgorodok, the agreement with Ministry of Science and High Education of the Russian Federation no.~075-15-2019-1613. The second author was supported by FAPDF and CNPq-Brazil.


\begin{thebibliography}{99}

    \bibitem{acc-khu-shu}    {C.~Acciarri, E.~I.~Khukhro and  P.~Shumyatsky},  Profinite groups with an automorphism whose fixed points are right Engel, \textit{Proc. Amer. Math. Soc.} \textbf{147}, no.~9 (2019), 3691--3703.

\bibitem{acc-shu16} C. Acciarri and P. Shumyatsky, Profinite groups and the fixed points of coprime automorphisms, \textit{J.~Algebra} \textbf{452} (2016), 188--195.

\bibitem{acc-shu-sil18} C. Acciarri, P. Shumyatsky, and D. S. Silveira, On groups with automorphisms whose fixed points are Engel,    \textit{Ann. Matem. Pura Appl.} \textbf{197} (2018), 307--316.

  \bibitem{acc-shu-sil19}  C. Acciarri, P. Shumyatsky, and D. Silveira,
Engel sinks of fixed points in finite groups, \textit{J.~Pure Appl. Algebra} \textbf{223}, no.~11 (2019), 4592--4601.

\bibitem{acc-sil18} C. Acciarri and D. Silveira,
Profinite groups and centralizers of coprime automorphisms whose elements are Engel, \textit{J.~Group Theory} \textbf{21}, no.~3 (2018), 485--509.

\bibitem{acc-sil20} C. Acciarri and D. Silveira, Engel-like conditions in fixed points of automorphisms of profinite groups, \textit{Ann. Matem. Pura Appl. (4)} \textbf{199}, no.~1 (2020), 187--197.

\bibitem{bah-zai} Yu. A. Bahturin and M. V. Zaicev, Identities of graded algebras, \textit{J.~Algebra} \textbf{205}, no.~1 (1998), 1--12.

\bibitem{bou} N. Bourbaki, Elements of Mathematics. Lie Groups
and Lie Algebras. Part I: Chapters 1--3, Hermann, Paris, Addison-Wesley, Reading, MA, 1975.

\bibitem{anal} J. D. Dixon,  M. P. F. du Sautoy, A. Mann, and  D. Segal,
\textit{Analytic pro-$p$ groups},  2nd ed., Cambridge Univ. Press, Cambridge, 1999.

\bibitem{gor}  D. Gorenstein,
\textit{Finite groups},  2nd ed., Chelsea, 1980.

\bibitem{hup} B. Huppert, \textit{Endliche Gruppen}. I,
Springer, Berlin, 1967.

\bibitem{kel} J. L. Kelley, \textit{General topology}, Grad. Texts in Math., vol.~27, Springer, New York,  1975.

     \bibitem{khu-shu99}    E. I. Khukhro and  P. Shumyatsky,  Bounding the exponent of a finite group with  automorphisms, \textit{J.~Algebra} \textbf{212}, no.~1 (1999),  363--374.

  \bibitem{khu-shu16}     {E.~I.~Khukhro and  P.~Shumyatsky},   Almost Engel finite and profinite groups, \textit{Int. J. Algebra Comput.}
{\bf  26}, no.~5 (2016), 973--983.

\bibitem{khu-shu18} E. I. Khukhro and P. Shumyatsky, Almost Engel compact groups, \emph{J.~Algebra} {\bf  500} (2018), 439--456.

         \bibitem{khu-shu18a}    {E.~I.~Khukhro and  P.~Shumyatsky},   Finite groups with Engel sinks of bounded rank, \textit{Glasgow Math.~J.} {\bf 60}, no.~3 (2018),   695--701.

     \bibitem{khu-shu19}    {E.~I.~Khukhro and  P.~Shumyatsky}, Compact groups   all elements of which are almost right Engel, \textit{Quart. J. Math} \textbf{70} (2019), 879--893.

     \bibitem{khu-shu20}    E. I. Khukhro and  P. Shumyatsky,  Compact groups with countable Engel sinks, \textit{Bull. Math. Sci.} \textbf{9}, no.~2 (2020) 2050015.

 \bibitem{khu-shu-tra}    {E.~I.~Khukhro,  P.~Shumyatsky, and G. Traustason}, Right Engel-type subgroups and length parameters of finite groups, \textit{J. Austral. Math. Soc.},   2019.


\bibitem{knu} D. E. Knuth, \textit{The Art of Computer Programming: Volume 1: Fundamental Algorithms}, 3rd ed.,  Addison-Wesley, Reading, MA, 1997.

             \bibitem{laz} M. Lazard, Groupes analytiques $p$-adiques, \textit{Publ. Math. Inst. Hautes \'Etudes
Sci.} \textbf{26} (1965), 389--603.

    \bibitem{lub-man2} A. Lubotzky and A. Mann,
Powerful p-groups. II: $p$-adic analytic groups, \textit{J.~Algebra} \textbf{105} (1987), 506--515.

\bibitem{rib-zal} L. Ribes and P. Zalesskii, \emph{Profinite groups},
Springer, Berlin, 2010.

\bibitem{rob} D. J. S. Robinson, A course in the theory of groups, Springer, New York, 1996.

    \bibitem{sha} A. Shalev, Polynomial identities in graded group rings, restricted Lie algebras and $p$-adic analytic groups, \textit{Trans. Amer. Math. Soc.} \textbf{337}, no.~1 (1993), 451--462.

\bibitem{war73} J. N. Ward, On groups admitting a noncyclic abelian automorphism group, \textit{Bull. Aust. Math. Soc.} \textbf{9} (1973), 363--366.

\bibitem{wil} J. S. Wilson, \emph{Profinite groups}, Clarendon Press, Oxford, 1998.

\bibitem{wi-ze} J. S. Wilson and E. I. Zelmanov, Identities for Lie algebras of pro-$p$ groups, \emph{J. Pure Appl. Algebra} {\bf 81}, no.~1 (1992), 103--109.


\bibitem{ze92} E. Zelmanov, Nil rings and periodic groups, \textit{Korean Math. Soc. Lecture
Notes in Math.}, Seoul, 1992.

\bibitem{ze95}    E. Zelmanov, Lie methods in the theory of nilpotent groups, in: \textit{Groups'\,93 Galaway/St Andrews}, Cambridge Univ. Press, Cambridge, 1995, 567--585.

\bibitem{ze17} E. Zelmanov,
Lie algebras and torsion groups with identity, \textit{J.~Comb. Algebra} \textbf{1}, no.~3 (2017), 289--340.

\end{thebibliography}
\end{document}